\documentclass{cmslatex}
\usepackage[paperwidth=7in, paperheight=10in, margin=.875in]{geometry}
 \usepackage[backref,colorlinks,linkcolor=red,anchorcolor=green,citecolor=blue]{hyperref}
\usepackage{amsfonts,amssymb}
\usepackage{amsmath}
\usepackage{amssymb} 
\usepackage{graphicx}
\usepackage{cite}
\usepackage{enumerate}
\usepackage{ulem}
\usepackage{tikz}
\usetikzlibrary{arrows.meta}
\renewcommand{\theequation}{\arabic{section}.\arabic{equation}}

\def\R{\mathbb{R}}
\def\E{\mathbb{E}}

\newcommand{\tcb}[1]{\textcolor{black}{#1}}

   \allowdisplaybreaks
\begin{document}
  \title{Propagation of chaos in path spaces via information theory}

        \author{Lei Li\thanks{School of Mathematical Sciences, Institute of Natural Sciences, MOE-LSC, Shanghai Jiao Tong University, Shanghai, 200240, P.R.China; Shanghai Artificial Intelligence Laboratory (leili2010@sjtu.edu.cn).}
        \and Yuelin Wang\thanks{School of Mathematical Sciences, Institute of Natural Sciences, MOE-LSC, Shanghai Jiao Tong University, Shanghai, 200240, P.R.China (sjtu$\_$wyl@sjtu.edu.cn).}
        \and Yuliang Wang\thanks{School of Mathematical Sciences, Institute of Natural Sciences, MOE-LSC, Shanghai Jiao Tong University, Shanghai, 200240, P.R.China (YuliangWang$\_$math@sjtu.edu.cn).}}

         \pagestyle{myheadings} \markboth{PROPAGATION OF CHAOS VIA INFORMATION THEORY}{LEI LI, YUELIN WANG, AND YULIANG WANG} \maketitle

          \begin{abstract}
               
              \tcb{Propagation of chaos for interacting particle systems has been an active research topic over decades. We propose an alternative approach to study the mean-field limit of the stochastic interacting particle systems via tools from information theory. In our framework, the propagation of chaos is reduced to the space for driving processes with possible lower dimension. Indeed, after applying the data processing inequality, one only needs to estimate the difference between the drifts of the particle system and the mean-field Mckean stochastic differential equation. This point is particularly useful in situations where the discrepancy in the driving processes is more apparent than the investigated processes. We will take the second order system as well as other examples for the illustration of how our framework could be used. This approach allows us to focus on probability measures in path spaces for the driving processes, avoiding using the usual hypocoercivity technique or taking the pseudo-inverse of the diffusion matrix, which might be more stable for numerical computation. Our framework is different from current approaches in literature and could provide new insight into the study of interacting particle systems. }
          \end{abstract}


\section{Introduction}\label{sec:intro}

The interacting particle system, mostly built upon basic physical laws including Newton's second law, has received growing popularity \tcb{recent} years in the study of both natural and social sciences. Practical application of such large-scale interacting particle systems includes groups of birds \cite{cucker2007emergent}, consensus clusters in opinion dynamics \cite{motsch2014heterophilious}, chemotaxis of bacteria \cite{horstmann20031970}, etc. Despite its strong applicability, the theoretical analysis and practical computation for the interacting particle system is rather complicated, mainly due to the fact that the particle number $N$ is very large in many practical settings. One classical strategy to reduce this complexity is to study instead the ``mean-field" regime. The limiting partial differential equation (mean-field equation) is used to describe the behavior of the particle system as $N \rightarrow \infty$. \tcb{This approximation allows one to obtain a one-body model instead of the original many-body one. For instance, Jeans proposed a mean-field equation to study the galactic dynamics in 1915 \cite{jeans1915theory}. Much work has been done to study the mean-field behaviors of various kinds of interacting particle systems \cite{georges1996dynamical,lasry2007mean,lu2023two,golse2016mean,natalini2020mean} in the past decades.}


\tcb{Here, let us take the second order system as the example to explain the concepts of mean field limit and propagation of chaos.} The second-order system is described by Newton's second law for $N$ point particles driven by 2-body interaction forces and Brownian motions, satisfying the following system of stochastic differential equations (SDE):
\begin{equation}\label{eq:particle}
\left\{
\begin{aligned}
d X_i(t)  &=V_i(t) d t,\\
md V_i(t)  &=\frac{1}{N-1} \sum_{j: j \neq i} K\left(X_i(t)-X_j(t)\right) d t - \gamma V_i(t) dt+\sigma \cdot  d W_i(t),\quad 1\leq i \leq N,
\end{aligned}\right.
\end{equation}
\tcb{where $m$ and $\gamma$ represent the mass $m$ and friction coefficient respectively, $X_i(t) , V_i(t)\in \mathbb{R}^d$.  The processes $W_i(t)$ ($1 \leq i \leq N$) are independent Brownian motions in $\mathbb{R}^{d^\prime}$, and $K$: $\mathbb{R}^d \rightarrow \mathbb{R}^d$ is the interaction kernel. We assume that the initial data $\{(X_i(0),V_i(0))\}$ are $i.i.d$ drawn from some initial law $F_0^N$ independent of the Brownian motions.} Denote $Z_i(t) := (X_i(t), V_i(t))$, and the corresponding joint law
\begin{equation}
    F^N_t\left(z_1, \cdots, z_N\right)=\operatorname{Law}\left(Z_1(t), \cdots, Z_N(t)\right) \in \mathcal{P}(\mathbb{R}^{2Nd}),
\end{equation}
where $\mathcal{P}(\mathbb{R}^{2Nd})$ denotes the probability measure space on $\mathbb{R}^{2Nd}.$
Then, the evolution of the density $F_t^N$ satisfies a Liouville's equation \cite{gibbs1879fundamental,gibbs1902elementary}:

\begin{multline}\label{eq:lioville}
    \partial_t F_t^N+\sum_{i=1}^N \nabla_{x_i}\cdot(v_i  F_t^N)+\frac{1}{m}\sum_{i=1}^N \nabla_{v_i} \cdot \left(\frac{1}{N-1} \sum_{j \neq i} K\left(x_i-x_j\right) F_t^N - \gamma v_i F_t^N\right) =\\
    \frac{1}{2m^2}  \sum_{i=1}^N \nabla^2_{v_i}:( \Lambda F_t^N),
\end{multline}
with $F_t^N|_{t=0} = F_0^N.$
Note that the matrix $\Lambda$ is defined by $\Lambda := \sigma \sigma^T$. \tcb{Here, ``$:$" means the Hilbert-Schmidt inner product so that
$\nabla^2_{v_i}:( \Lambda F_t^N)=\sum_{j,k} \partial_{v_j v_k}^2 (\Lambda_{jk}F_t^N)$. As the particle number $N$ tends to infinity, the correlation between any two focused particles through the weak interaction is expected to vanish. Hence,  
if two particles are initially independent, then they are expected to be independent as $N\to\infty$ at any fixed time point $t>0$. This is the so-called propagation of chaos. Due to the asymptotic independence, a fixed particle with position and velocity $\bar{Z}_i(t):=(\bar{X}_i(t), \bar{V}_i(t))$ is then expected to satisfy the following mean field Mckean SDE system:
\begin{equation}\label{eq:mckeansde}
    d\bar{X}(t)  =\bar{V}(t) d t,\quad
md \bar{V}(t)  =K {*} \bar{\rho}_t (\bar{X}(t)) d t - \gamma \bar{V}(t) dt +\sigma \cdot d W(t),
\end{equation}
where $\bar{F}_t \in \mathcal{P}(\mathbb{R}^{2d})$ is the law, and $\bar{\rho}_t(x) := \int_{\mathbb{R}^d} \bar{F}_t(x,v) dv$ is its marginal. 
The law $\bar{F}_t$ is then expected to satisfy the following mean field kinetic Fokker-Planck equation \cite{jabin2014review,jabin2017mean}:
\begin{equation}\label{eq:meanfieldpde}
    \partial_t \bar{F}_t +\nabla_x\cdot(v \bar{F}_t) + \frac{1}{m} \nabla_v \cdot \left( K {*} \bar{\rho}_t \bar{F}_t - \gamma v \bar{F}_t \right)  = \frac{1}{2m^2}  \nabla^2_v :(\Lambda  \bar{F}_t),\quad \bar{F}_t|_{t=0} = \bar{F}_0.
\end{equation}
Rigorous justification of this mean limit, or the propagation of chaos, has then become an active research topic.
}

The prevalent method in analyzing mean-field limits is based on Dobrushin's Estimate, which is proposed in 1979 by Dobrushin etc. \cite{dobrushin1979vlasov}, to study the stability of the mean-field characteristic flow in terms of Wasserstein distances. Dobrushin-type analysis has now been a classical tool in mean-field limits for Valsov-type equations during these decades.
Based on \tcb{Dobrushin-type} analysis, one can then prove the mean-field limit for the deterministic system in a finite time interval $[0,T]$ in terms of Wasserstein distances \cite{braun1977vlasov,neunzert2006approximation,golse2016mean}. 
Another way is to compare the stochastic trajectories through certain coupling technique.
By considering trajectory controls, the mean-field limit for stochastic systems with Lipschitz kernel $K$ has been established \cite{sznitman1991topics,graham1996asymptotic,guillin2021kinetic}. 

\tcb{Another class of methods is to compare the laws directly. What has become popular recently on chaos qualification is given by the analysis of relative entropy (also called Kullback-Leibler divergence, KL-divergence) between $F_t^{N:k}=\int_{(\R^{2d})^{N-k}}F_t^N dz_{k+1}\cdots dz_N$ and $k$ tensorized product of $\bar{F}_t$, $\bar{F}_t^{\otimes k}:=\prod_{i=1}^k \bar{F}_t(z_i)$ for $1\le k\le N$. The analysis could also be performed on the laws on path space with $F_t^{N:k}$ and $\bar{F}_t^{\otimes k}$ being their time marginals. Some early results in path space using the relative entropy have been achieved in the last century (e.g. \cite{arous1990methode,arous1999increasing}). For time marginal distributions, Jabin et. al. proved the propagation of chaos for Vlasov-type systems with $\mathcal{O}(k/N)$ bound, assuming the interaction kernel $K$ is bounded, and the propagation of chaos for first order systems with singular kernels \cite{jabin2018quantitative}.  For results in path space, Lacker obtained the propagation of chaos relying on Girsanov's and Sanov's theorem \cite{lacker2018strong} and the BBGKY hierarchy \cite{lacker2023hierarchies,lacker2023sharp}. The approach in \cite{lacker2023hierarchies,lacker2023sharp} yields an $\mathcal{O}((k/N)^2)$ bound of the relative entropy between the marginal law of $k$ particles and its limiting product measure. 
For singular $L^p$-interactions, Toma{\v{s}}evi{\'c} et. al. used the the partial Girsanov transform to derive the propagation of chaos in \cite{jabir2018mean,tomavsevic2023propagation}. Recently, Hao et. al. further showed the strong convergence of the propagation of chaos with singular $L^p$-interactions in  \cite{hao2024strong}. Also, based on Lacker's approach, Cattiaux gave an $\mathcal{O}(k/N)$ estimate on the path space in \cite{cattiaux2024entropy}, by using the invariance of relative entropy under time reversal \cite{cattiaux2013singular}. The results in \cite{delarue2021uniform} and \cite{guillin2024uniform} are uniform in time for the Coulomb and the Biot-Savart kernel, respectively. There is a vast literature on this topic, and we provide recent review articles \cite{chaintron2022propagation,chaintron2022propagation2} for the convenience of readers.}

In this work, \tcb{we propose to use the information theory to study the propagation of chaos by comparing the discrepancy between the joint law of the particle system and the corresponding mean-field equation in terms of KL-divergence defined by
\begin{equation} D_{KL}\left(P \| Q \right) :=
\left\{\begin{aligned}
    \int_E \log&\frac{dP}{dQ} dP, \quad& P\ll Q,\\
    &\infty, \quad& \text{otherwise,}
\end{aligned}\right.
\end{equation}
where $P$ and $Q$ are two probability measures over some appropriate space $E$. In our framework, the propagation of chaos is reduced to the space for driving processes with possible lower dimension. We will mainly take the second-order systems as the example, which avoids using the usual hypocoercivity technique or taking the pseudo-inverse of the diffusion matrix. We remark that the bounds under relative entropy for the second order system can be obtained by direct Girsanov transform if one takes the pseudo-inverse of the degenerate diffusion matrix as mentioned in \cite[Remark 4.5]{lacker2023hierarchies}. Nevertheless, we believe our approach is still of significance as there is no degeneracy in diffusion if we look at the measures in the space for driving processes, which could be more stable for numerical computation.  We will also look at the application of our framework to other illustrating examples.}

\tcb{
We focus an estimate for the KL-divergence between the laws in path space, in particular $D_{KL}(F^N_{[0,T]} \| \bar{F}^{\otimes N}_{[0,T]})$. Here $F^N_{[0,T]}$ and $\bar{F}^{\otimes N}_{[0,T]}$ are probability distributions in the path space $\mathcal{X} := C([0,T];\mathbb{R}^{2Nd})$ (for fixed time interval $[0, T]$) corresponding to the SDE systems \eqref{eq:particle} and $N$ independent  copies of \eqref{eq:mckeansde} respectively. Denoting $\mathcal{Z}_{[0,T]} := (Z_1, \dots, Z_N)_{[0,T]}$, $\bar{\mathcal{Z}}_{[0,T]} := (\bar{Z}_1, \dots, \bar{Z}_N)_{[0,T]}$ in the path space, the path measures satisfy $F^N_{[0,T]} = \mathcal{Z}_{[0,T]} {\#}\mathbb{P}$, and $\bar{F}^{\otimes N}_{[0,T]} = \bar{\mathcal{Z}}_{[0,T]}{\#}\mathbb{P}$  ($\mathbb{P}$ is the original probability measure such that $W$ is a Brownian motion). With this setting, $F_t^N$ is the time marginal of $F^N_{[0,T]}$, and $\bar{F}_t^{\otimes N}$ is then the time marginal of $\bar{F}^{\otimes N}_{[0,T]}$. We then regard the process of the mean-field McKean SDEs and the interacting particle systems as the same dynamical system with different driving processes (input signals). Then, applying the data processing inequality, we can work on probability measures in the space for the input signals instead of the space for the particles. The former space is sometimes easier to deal with than the latter as one may avoid the degeneracy of the diffusion. Moreover, the dimension could be lower. This has similarity with the so-called latent space in machine learning \cite{liu2019latent}.} \tcb{Moreover, we will also present the applications of the framework onto neural networks and numerical analysis to illustrate this point.}

\tcb{The rest of the paper is organized as follows: In Section \ref{sec:sec2}, we present our main ideas. The result (Theorem \ref{thm:main}) on the propagation of chaos for the second-order system in path space is shown in Section \ref{sec:mainresult} for both bounded kernels (not necessarily smooth) or Lipschitz kernels (not necessarily bounded) with the necessary assumptions and auxiliary lemmas. In Section \ref{sec:discussion}, we provide two applications of our approach on numerical analysis and neural networks. Lastly in Section \ref{sec:dis}, we perform a discussion on the reversed relative entropy and mass-independence.}

\section{The main idea of the new framework}\label{sec:sec2}

In this section, \tcb{taking the second order system as the example, we present} the main ideas \tcb{without rigorous proof. The rigorous mathematical setup, assumptions and proof will be given in the next section.}

For fixed $[0,T]$, let $\bar{F}_{[0,T]}$ be the law of the trajectories of the following Mckean SDE system \eqref{eq:mckeansde}.
Then the tensorized distribution $\bar{F}^{\otimes N}_{[0,T]}$ is the law of trajectories of the following system:
\begin{equation}\label{eq:mckeansde for N}
    d\bar{X}_i(t)  =\bar{V}_i(t) d t,\quad
md \bar{V}_i(t)  =K {*} \bar{\rho}_t (\bar{X}_i(t)) d t-\gamma \bar{V}_i(t) dt+\sigma \cdot d W_i(t), \quad 1 \leq i \leq N,
\end{equation}
and the particles $\bar{Z}_i := (\bar{X}_i, \bar{V}_i)$, $1 \leq i \leq N$ are independent.

The key idea of this work is rewriting \eqref{eq:particle} above into:
\begin{equation}\label{eq:mckeansde_rewrite}
    dX_i(t)  =V_i(t) d t,\quad
md V_i(t)  =K {*} \bar{\rho}_t(X_i(t)) d t - \gamma V_i(t) dt+  d \theta_i^{(1)}(t), \quad 1 \leq i \leq N,
\end{equation}
where the process $\theta_i^{(1)}(t)$ is defined by
\begin{align} \label{eq:N_i^1}
\theta_i^{(1)}(t) :&=  \int_0^t \left(  \frac{1}{N-1} \sum_{j: j\neq i} K(X_i(s) - X_j(s)) - K {*} \bar{\rho}_s (X_i(s))\right) ds +\sigma\cdot W_i(t)\notag\\
&= \int_0^t b_i(s, X(s))\,ds+\sigma\cdot W_i(t).
\end{align}
Here, 
\begin{gather}
b_i(s, x):=\frac{1}{N-1}\sum_{j: j\neq i}K(x_i-x_j)-K*\bar{\rho}_s(x_i).
\end{gather}
We also denote
\begin{equation}\label{eq:theta2def}
\theta_i^{(2)}(t)=\sigma\cdot W_i(t).
\end{equation}
\tcb{Based on \eqref{eq:mckeansde_rewrite} and \eqref{eq:mckeansde for N}, formally, we write the generalized dynamics
\begin{equation}\label{eq:box}
    d\hat{X}_i(t)  =\hat{V}_i(t) d t,\quad
md \hat{V}_i(t)  =K {*} \bar{\rho}_t(\hat{X}_i(t)) d t - \gamma \hat{V}_i(t) dt+  d\theta_i(t), \quad 1 \leq i \leq N.
\end{equation}
Here, $\theta:=(\theta_1,\cdots,\theta_N)$ is a \tcb{driving process}. In \eqref{eq:mckeansde_rewrite}, the \tcb{driving process} is taken as the noise process $\theta^{(2)},$ while in \eqref{eq:mckeansde for N} is taken as $\theta^{(1)}.$ For fixed initial data, as shown in \eqref{illusion:dpi}, the \tcb{driving process} $\theta$ can be viewed as an input, then through the equation \eqref{eq:box}, the particle trajectory is obtained as an output.
    \vskip2mm
    \begin{equation}\label{illusion:dpi}
        \begin{tikzpicture}    
        \draw[->] (5,0)--(6,0);
        \draw (6,-0.45) rectangle (8,0.45);
        \draw[->] (8,0)--(9,0);
        \node (A) at (3.6,0) {\tcb{driving process} $\theta$};
        \node (C) at (7,0) {\eqref{eq:box}};
        \node (D) at (10.7,0) {trajectory $(X,V)$};
    \end{tikzpicture} 
    \end{equation}
From this perspective, a natural guess is that, if there is only slight difference between two \tcb{driving processes}, the difference between the outputs might be not large.}
\tcb{Luckily, if the mean field McKean SDE \eqref{eq:mckeansde} has pathwise uniqueness, the following well-known data processing inequality \cite{thomas1992information} can help to establish such intuition.}

\vskip2mm
\begin{lemma}[data processing inequality]\label{lmm:dp}
    \tcb{Consider a given conditional probability $P_{Y \mid X}$} and that $Y$ is produced by $P_{Y \mid X}$ given $X$. If $P_Y$ is the distribution of $Y$ when $X$ is generated by $P_X$, and $Q_Y$ is the distribution of $Y$ when $X$ is generated by $Q_X$, then for any convex function $f: \mathbb{R}^{+} \rightarrow \mathbb{R}$ satisfying $f(1)=0$ and being strictly convex at $x=1$, it holds
\begin{equation}
    D_f\left(P_Y \| Q_Y\right) \leq D_f\left(P_X \| Q_X\right),
\end{equation}
where the $f$-divergence $D_f(\cdot \| \cdot)$ is defined by
\tcb{\begin{equation}
    D_f(P \| Q):=\left\{\begin{aligned}
        \mathbb{E}_Q&\left[f\left(\frac{d P}{d Q}\right)\right] \quad &P\ll Q,\\
        &\infty \quad &\text{otherwise}.
    \end{aligned}\right.
\end{equation}}
\end{lemma}
\tcb{\begin{remark}
    Taking $f(x)=x\log x$, the $f$-divergence $D_f$ is the famous KL-divergence. In this paper, we focus on this special case.
\end{remark}}

\tcb{\begin{remark}
    The data processing inequality is also well-known in probability and statistics (e.g. \cite{lacker2023hierarchies}), which states that $D_{KL} (\nu\circ g^{-1}\| \nu^\prime \circ g^{-1})\le D_{KL} (\nu \circ \nu^\prime)$ for any probability measures $\nu,\,\nu^\prime$ on a common measurable space and any measurable function $g$ into another measurable space.
\end{remark}}
\vskip2mm

\tcb{Now, by the data processing inequality, we can control the KL-divergence between the output into that between the input. In this respect, we change our problem from the trajectory space into the space for the driving process $\theta$.}
Exactly, we find that
\[
D_{KL}(F^N_{[0,T]} \| \bar{F}^{\otimes N}_{[0,T]}) \le D_{KL}(Q^1 \| Q^2),
\]
where we recall $F^N_{[0,T]}$ and $\bar{F}^{\otimes N}_{[0,T]}$ are path measures introduced in Section \ref{sec:intro} and we denote $Q^j$ to be the path measures for 
\[
\theta^{(j)}:=(\theta_1^{(j)}, \cdots, \theta_N^{(j)}(t)).
\]
To compute the latter relative entropy, we rewrite the equation for $\theta^{(1)}$ by
\begin{gather}
\theta_i^{(1)}=\int_0^t b_i(s, X(s))\,ds+\sigma\cdot W_i(t)
=:\int_0^t \tilde{b}_i(s, [\theta^{(1)}]_{[0, s]})\,ds +\sigma \cdot W_i(t).
\end{gather}
Then, $\theta^{(1)}$ satisfies an SDE in the space of the \tcb{driving process}, with a dimension smaller than that of  $(X, V)$. 
\tcb{Then, by Girsanov's transform, it holds} 
\begin{gather}\label{eq:KLQaux1}
D_{KL}(Q^1 \| Q^2)
= -\mathbb{E} \log \frac{dQ^2}{dQ^1}[\theta^{(1)}]
=\tcb{\frac{1}{2}\mathbb{E} \sum_i \int_0^T \langle b_i(s, X(s)), (\sigma\sigma^T)^{-1}b_i(s, X(s))\rangle ds.}
\end{gather}
\tcb{Note that this reduction avoids the degeneracy of the diffusion coefficient. Though the degeneracy can be 
treated by using the pseudo-inverse as remarked in \cite{lacker2023hierarchies}, such a reduction could be helpful for practical estimates using numerical computations.
We will give more details in the next sections.}

\tcb{Let us discuss the choice of the noise and dynamical system. }
One may be tempted to rewrite the mean-field McKean SDE into
\begin{equation*}
    d\bar{X}_i  =\bar{V}_i d t,\quad
md \bar{V}_i  =\frac{1}{N-1} \sum_{j:j\neq i}K(\bar{X}_i - \bar{X}_j) d t - \gamma \bar{V}_i dt +  d \eta_i^{(2)}, \quad 1 \leq i \leq N,
\end{equation*}
with
\begin{equation*} 
    \eta_i^{(2)}(t) :=  \int_0^t \left(   K {*} \bar{\rho}_s (\bar{X}_i)- \frac{1}{N-1} \sum_{j: j\neq i} K(\bar{X}_i - \bar{X}_j)\right) ds +\sigma\cdot W_i(t).
\end{equation*}
Then, \tcb{the} $N$-body interacting particle system is given by
\begin{equation*}
d X_i  =V_i d t,\quad
md V_i  =\frac{1}{N-1} \sum_{j:j\neq i}K(X_i - X_j) d t - \gamma V_i dt + d \eta_i^{(1)},\quad 1\leq i \leq N,
\end{equation*}
with $\eta_i^{(1)}(t) := \sigma\cdot  W_i(t)$ $(1 \leq i \leq N)$. 

The two systems are also the same dynamical system with difference driving noises 
\[
\eta^{(j)}(\cdot):=(\eta_1^{(j)}(\cdot), \cdots, \eta_N^{(j)}(\cdot)).
\]
\tcb{At first glance,} this formulation seems good since the drift in $\eta^{(2)}$ involves only the 
solution to the mean-field McKean SDE. Then, one may apply the law of large numbers. 
However, this is not the case. In fact,  applying the data processing inequality, one has
\[
D_{KL}(F^N_{[0,T]} \| \bar{F}^{\otimes N}_{[0,T]}) \le D_{KL}(\bar{Q}^1 \| \bar{Q}^2),
\]
where $\bar{Q}^j$ is the law for $\eta^{(j)}$.
\tcb{We consider
\[
\eta_i^{(2)}:=-\int_0^t b_i(s, \bar{X}(s))\,ds+\sigma\cdot W_i(t)
=-\int_0^t b_i(s, \pi_s\circ \hat{\Phi}_s(\eta^{(2)}))\,ds+\sigma\cdot W_i(t).
\] }
Here, the mapping $ \hat{\Phi}_s : \eta \mapsto (X, V)$
is the solution map \tcb{for} the $N$-body interacting dynamical system and $\pi_s f=f(s)$ is the time marginal.
\tcb{This is again an SDE in the space for the driving process.} Then,
\begin{equation*}
D_{KL}(\bar{Q}^1 \| \bar{Q}^2) = \mathbb{E}_{X \sim \bar{Q}^1}\left[-\log \frac{d\bar{Q}^2}{d\bar{Q}^1}(X) \right].
\end{equation*}
The point is that the Radon-Nykodym derivative is integrated \tcb{against} $\bar{Q}^1$. The Girsanov's transform then gives that
\tcb{\begin{gather}
\begin{aligned}
    \mathbb{E}_{X \sim \bar{Q}^1}\left[-\log \frac{d\bar{Q}^2}{d\bar{Q}^1}(X) \right]&
=\sum_i \mathbb{E}\int_0^t \frac{1}{2}\langle b_i(s, \pi_s\circ\hat{\Phi}_s(\eta^{(1)})),\Lambda^{-1}b_i(s, \pi_s\circ\hat{\Phi}_s(\eta^{(1)}))\rangle \,ds\\
&=\sum_i \mathbb{E}\int_0^t \frac{1}{2}\langle b_i(s, X(s)),\Lambda^{-1}b_i(s, X(s))\rangle\,ds,
\end{aligned}
\end{gather}}
where the inside is changed from $\eta^{(2)}$ to $\eta^{(1)}$!
The eventual result is the same as \eqref{eq:KLQaux1}.

\section{The application to the second order systems}\label{sec:mainresult}

In this section, we establish the propagation of chaos in path space for the second order systems using the framework of information theory, in particular the data processing inequality.

\tcb{We first present our assumptions on the kernels and coefficients. The first set of assumptions requires that $K$ is bounded. }
\vskip2mm
\begin{assumption}\label{ass:K}
\quad
    \begin{itemize}
        \item[(a)] The kernel $K$ has finite essential bound, namely, $\| K \|_{L^{\infty}(\mathbb{R}^d)} < +\infty$. 
        \item[(b)] The matrix $\Lambda = \sigma \sigma^T$ is non-degenerate with minimum \tcb{eigenvalue} $\lambda>0.$
    \end{itemize}
\end{assumption} 
\vskip1mm
\tcb{\begin{remark}
    In our main text, the matrix $\sigma$ is a constant matrix for notational convenience. However, a time- and state-dependent diffusion $\sigma(t,X_i,V_i)$ is allowed as long as the spectrum of $\Lambda = \sigma\sigma^T$ is uniformly bounded above and away from zero and the well-posedness results in the following subsection preserves. It is similar with \cite[Remark 4.5]{lacker2023hierarchies}.
\end{remark}}
 \vskip2mm
\tcb{The boundedness condition for the interaction kernel $K$ (condition (a) in Assumption \ref{ass:K} above) sometimes is strong in practice.
Here, if we assume that the initial distribution has a fast decaying tail, we can allow a Lipschitz kernel. In fact, we will assume also alternatively the followings:}
\vskip2mm
\begin{assumption}\label{ass:new}
\quad
    \begin{itemize}
        \item[(a)] \tcb{The initial space-marginal distribution of the Mckean SDE \eqref{eq:mckeansde} is sub-Gaussian, namely, there exists $C>0$ such that for any $a \geq 0$, $P(|\bar{X}_1(0)| > a) \leq 2 \exp(-a^2 / C^2)$.}
        \item[(b)] \tcb{The interaction kernel $K(\cdot)$ is $L_K$-Lipschitz, namely, $\forall x,y \in\mathbb{R}^d$, $|K(x) - K(y)| \leq L_K |x-y|$.}
        \item[(c)] The matrix $\Lambda = \sigma \sigma^T$ is non-degenerate with minimum \tcb{eigenvalue} $\lambda>0.$
    \end{itemize}
\end{assumption}

\tcb{\subsection{The well-posedness of the mean field McKean SDE.}}

\tcb{Under either Assumption \ref{ass:K} or \ref{ass:new}, we are able to establish the propagation of chaos using nearly the same method. As a first step, we consider the solution map of \eqref{eq:mckeansde}.  For {\it fixed initial data}, we rewrite it as}
\begin{equation}\label{eq: temphat}
\begin{split}
&\hat{X}_i(t)  =\hat{X}_i(0)+\int_0^t\hat{V}_i(s) d s,\\
&m\hat{V}_i(t)  =m\hat{V}_i(0)+\int_0^t K {*} \bar{\rho}_s(\hat{X}_i(s)) d s - \gamma \int_0^t\hat{V}_i(s) ds+ \hat{\theta}_i(t),\quad 1\leq i \leq N.
\end{split}
\end{equation}
\tcb{We first have the following observation.}

\tcb{\begin{lemma}
Suppose that either Assumption \ref{ass:K} or Assumption \ref{ass:new} holds. Then, 
the mean field nonlinear kinetic Fokker-Planck equation \eqref{eq:meanfieldpde} has a unique solution that is 
in $C([0, T]; \mathcal{P}(\R^d))$ where the topology is the weak convergence of measures. Moreover, the solution is smooth 
for any $t>0$.
\end{lemma}}
\vskip2mm

\tcb{The result under Assumption \ref{ass:new} is very standard because the corresponding SDE system even has strong solutions.
For the first, the well-posedness under some more general singular kernels have been established as well.
One may refer to \cite{krylov2005strong, zhang2010stochastic, hao2024strong} for related discussion.}

\tcb{As soon as we have the well-posedness for the nonlinear Fokker-Planck equation, then $K*\bar{\rho}_t$ is smooth for any $t>0$,
and thus locally Lipschitz. Now, we take $t\mapsto \bar{\rho}_t$ as given. We conclude the following.}

\tcb{\begin{lemma}
Suppose that either Assumption \ref{ass:K} or Assumption \ref{ass:new} holds. Then, the following integral equation has a unique continuous solution.
\begin{gather}
\begin{split}
& X(t)=X_0+\int_0^t V(s)\,ds,\\
& mV(t)=mV_0+\int_0^t K*\bar{\rho}_s(X(s))\,ds-\gamma\int_0^t V(s)\,ds+\eta(t),
\end{split}
\end{gather}
where $t\mapsto \eta(t)$ is a given continuous driving signal.
\end{lemma}}
\vskip2mm

\tcb{For the uniqueness, it is relatively straightforward. In fact, for any two continuous solutions and given $T>0$, they stay in a compact set. On this compact set, $K*\bar{\rho}_t$
is Lipschitz on $[\epsilon, T]$ for any $\epsilon>0$. The integral on $[0, \epsilon]$ can be made arbitrarily small.
The uniqueness can then be obtained by direct comparison.
For the existence, one may consider the regularized equation where $\bar{\rho}_t$ is redefined to be $\bar{\rho}_{\epsilon}$ for $t\in [0, \epsilon]$.
The obtained solution $(X^{\epsilon}(t), V^{\epsilon}(t))$ can be shown to be uniformly bounded. Then, it is not hard to show they are relatively compact in $C([0, T]; \R^d)$ by the Arzela-Ascoli criterion, with any limit point being a solution of the integral equation.}

With the above fact, the mean field McKean SDE \eqref{eq:mckeansde} actually has a unique strong solution. For a fixed time $t$, we may introduce the mapping
\begin{equation}\label{eq:PhiTm}
    \Phi_t : \quad \hat{\theta} \mapsto \hat{\mathcal{Z}} := (\hat{Z}_1, \dots, \hat{Z}_N),
\end{equation}
where $\hat{\theta} = (\hat{\theta}_1,\dots,\hat{\theta}_N)\in C([0, t]; \mathbb{R}^{Nd})$ is a generic \tcb{driving process}, $\hat{Z}_i (\cdot) := (\hat{X}_i (\cdot), \hat{V}_i (\cdot)),$ and $\hat{\mathcal{Z}}\in C([0, t]; \mathbb{R}^{2Nd})$ is the \tcb{solution} of the dynamical system \eqref{eq: temphat}.

For fixed $t$, $\Phi_t$ only depends on $\theta_s$ for $s\le t$.  If we change $t$, the solution process will clearly agree on the common subinterval. Below, we will consider varying $t$, but we will not change the notation $\hat{\theta}$ for convenience.
Moreover, the dependence on the initial data is also not written out explicitly for clarity.
Consequently, recalling the definitions $\mathcal{Z}_{[0,T]} = (Z_1, \dots, Z_N)$, $\bar{\mathcal{Z}}_{[0,T]} = (\bar{Z}_1, \dots, \bar{Z}_N),$ and $Z_i(t) = (X_i(t), V_i(t))$, $\bar{Z}_i(t) = (\bar{X}_i(t), \bar{V}_i(t))$, then  one has 
\begin{equation}\label{eq:mapping}
\mathcal{Z}_{[0,T]} = \Phi_T(\theta^{(1)}_{[0,T]}), \quad  \bar{\mathcal{Z}}_{[0,T]} = \Phi_T(\theta^{(2)}_{[0,T]}).
\end{equation}
With the conditions above, next we establish the propagation of chaos result  for distributions starting \tcb{from a} chaotic configuration (i.e., $F_0^N=\bar{F}_0^{\otimes N}$).

\subsection{Propagation of chaos in path space and the corollaries.}

\tcb{We again note a fact from standard SDE theory.}

\tcb{\begin{lemma}
Suppose that either Assumption \ref{ass:K} or Assumption \ref{ass:new} holds. The interacting particle system
\eqref{eq:particle} has a weak solution unique in law.
\end{lemma}}
\vskip2mm
\tcb{The existence of weak solution for bounded $K$ follows from a standard Girsanov transform \tcb{(see e.g. \cite[Theorem 8.6.5]{oksendal2013stochastic}, \cite[Theorem 27.1]{rogers2000diffusions}, \cite[Theorem 2.1]{leonard2012girsanov}).} 
The uniqueness in law for bounded kernels is also standard and one may refer to the discussion in \cite[page 155, Chapter 4, Section 18]{rogers2000diffusions}.}

\tcb{The weak well-posedness of the SDE implies that the Liouville equation \eqref{eq:lioville} has weak solutions. 
The uniqueness of the Liouville equation \eqref{eq:lioville} can also be established with the bounded or Lipschitz assumption on $K$ (see e.g. \cite{rockner2010weak}). It is straightforward to see that if the initial $F^N_0$ is symmetric, $F^N$ is symmetric due to the fact that $t \rightarrow F^N_t(p(z))$ satisfies the same Liouville equation as $t \rightarrow F^N_t(z)$, where $p(z)$ is an arbitrary permutation for $z \in (\mathbb{R}^{2d})^N$ (see, for instance,  a similar argument in \cite{muntean2016macroscopic}). Similar argument also applies to the law in the path space. In fact, for any weak solution $Z$, it is not hard to see $p(Z)$ is also a weak solution. Then, the uniqueness in law implies that the law in the path space is symmetric. This in fact arises from the exchangeability of the particle systems.}

\tcb{Next, we have the following result under Assumption \ref{ass:new}.}
\vskip2mm
\begin{lemma}
\tcb{Suppose that Assumption \ref{ass:new} holds. Then, the following statements hold.}
\begin{enumerate}
    \item \tcb{For any $t \in [0,T]$, the solution of the mean field McKean SDE \eqref{eq:meanfieldpde} is sub-Gaussian. }
    \item \tcb{The interaction kernel $K(\cdot)$ and the marginal distribution $\bar{\rho}_t$ of the McKean SDE \eqref{eq:mckeansde} satisfy: there exist $C>0$ such that $\forall x,y \in \mathbb{R}^d$ and $t \in [0,T]$, $|K(x-y) - K {*} \bar{\rho}_t(x)| \leq C(1 + |y|)$. }
\end{enumerate}
\end{lemma}
\vskip2mm
\tcb{The first claim can be verified by calculating $\E \,\exp(c(|\bar{X}|^2+|\bar{V}|^2))$ via It\^o's formula. The second one is actually also obvious
by the first-order moment bound for $\bar{X}(t)$, which is obvious under Assumption \ref{ass:new}.
Below, we present and prove the main result in this section.}
\vskip2mm
\begin{theorem}\label{thm:main}
    For fixed time interval $[0,T]$, assume that \tcb{either Assumption \ref{ass:K} or Assumption \ref{ass:new} holds. Consider the path measure $F^N_{[0,T]}$  for the weak solution to the second-order system \eqref{eq:particle}, with initial law $F^N_0 = \bar{F}^{\otimes N}_0$.}  Then, there exists a constant $C$ such that
    \begin{equation}\label{eq:thm1result}
        D_{KL}\left(F^N_{[0,T]} \| \bar{F}^{\otimes N}_{[0,T]}\right) \leq Ce^{CT}.
    \end{equation}
 Consequently,  for $1 \leq k \leq N$,
    \begin{equation}\label{eq:kthmarginal}
        D_{KL}\left(F^{N:k} \| \bar{F}^{\otimes k} \right) \leq Ce^{CT} \frac{k}{N}.
    \end{equation}
\end{theorem}
\vskip2mm
\begin{proof}
Recall \tcb{equations} \eqref{eq:mckeansde for N}-\eqref{eq:theta2def}. Note that we consider the weak solution to \eqref{eq:particle}.
Hence, the Brownian motions are not necessarily in the same space. However, since the McKean SDE has a strong solution, we may without loss of generality to take the Brownian motions in \eqref{eq:mckeansde for N} to be the ones used for the weak solutions of \eqref{eq:particle}, without altering the laws.

The corresponding \tcb{driving process} in the path space are $$\theta^{(j)}_{[0,T]} := \left(\theta^{(j)}_1(\cdot),\dots, \theta^{(j)}_N(\cdot)\right)_{0\le t\le T}\in C([0, T]; \mathbb{R}^{Nd}) \text{ for } j = 1,2.$$

Let $F^N_{[0,T]}(\cdot|z)$ denote the law of 
$\mathcal{Z}_{[0,T]} =(Z_1,\cdots, Z_N)$ (recall that $Z_i=(X_i, V_i)$) with initial data $\mathcal{Z}(0)=z\in \mathbb{R}^{Nd}$ and $\bar{F}^N_{[0,T]}(\cdot|z)$ is similarly defined. 
Then, for initial data obeying the distribution $\Bar{F}_0^{\otimes N}$, one has
\begin{gather}\label{eq:lawmixture}
F^N_{[0,T]}=\int_{\mathbb{R}^{Nd}} F^N_{[0,T]}(\cdot|z) \Bar{F}_0^{\otimes N}(dz),
\quad \bar{F}^{\otimes N}_{[0,T]}=\int_{\mathbb{R}^{Nd}} \bar{F}^N_{[0,T]}(\cdot|z) \Bar{F}_0^{\otimes N}(dz).
\end{gather}
By the data processing inequality (Lemma \ref{lmm:dp}), one has that
\begin{equation}\label{eq:afterdataprocess}
    D_{KL}(F^N_{[0,T]}(\cdot|z) \| \bar{F}^{N}_{[0,T]}(\cdot| z)) \leq D_{KL}(Q^1 \| Q^2) = \mathbb{E}_{X \sim Q^1}\left[-\log \frac{dQ^2}{dQ^1}(X) \right],
\end{equation}
where $Q^1$, $Q^2$ are path measures generated by $\theta^{(1)}_{[0,T]}$ and $\theta^{(2)}_{[0,T]}$ , respectively, corresponding to the time interval $[0,T]$. Namely, $Q^1 = \theta^{(1)}_{[0,T]}{\#} \mathbb{P}$, and $Q^2 = \theta^{(2)}_{[0,T]}{\#} \mathbb{P}$. 
By definition of the process $\theta^{(1)}_{[0,T]}$, $\theta^{(2)}_{[0,T]}$, $Q^2 \ll Q^1$ and the Radon-Nikodym derivative $\frac{dQ^2}{dQ^1}$ exists. 
One can find the expression of this Radon-Nikodym derivative explicitly by Girsanov's transform.  In fact, denote the \tcb{$Nd$-dimensional} vector $\boldsymbol{b}(s, x) = (\boldsymbol{b}^T_1,\cdots,\boldsymbol{b}^T_N)^T$ with
\begin{equation*}
    \boldsymbol{b}_i(s, x) :=\sigma^T \Lambda^{-1} \left( K {*} \rho_s (x_i) - \frac{1}{N-1} \sum_{j: j\neq i} K(x_i- x_j)\right).
\end{equation*}
Note that 
\[
\boldsymbol{b}(s, X(s)) = \boldsymbol{b}(s,  \pi_s\circ \Phi_s(\theta^{(1)}_{[0,s]})) =: \tilde{\boldsymbol{b}}(s, 
[\theta^{(1)}]_{[0,s]}),
\]
 where $\Phi_s$ is defined in \eqref{eq:PhiTm}, and $\pi_s$ maps $X_{[0,s]}$ in  path space to its time marginal, namely, $\pi_s(X_{[0,s]}) = X_s$.  Then the Girsanov's transform asserts that the Radon-Nikodym derivative in the path space satisfies 
\begin{align}\label{eq:aftergirsanov}
    \frac{dQ^2}{dQ^1}(\theta^{(1)}(\omega)) &= \exp\Big(\int_0^T   \tilde{\boldsymbol{b}}(s, [\theta^{(1)}]_{[0,s]}) \cdot  dW_s    - \frac{1}{2} \int_0^T \left| \tilde{\boldsymbol{b}}(s, [\theta^{(1)}]_{[0,s]}) \right|^2 ds\Big)\notag\\
    &= \exp\Big(\int_0^T   \boldsymbol{b}(s, X(s)) \cdot  dW_s    - \frac{1}{2} \int_0^T \left|\boldsymbol{b}(s, X(s)) \right|^2 ds\Big).
\end{align}
\tcb{In Appendix \ref{app:girsanov}, we present a formal derivation of the details for \eqref{eq:aftergirsanov}. The strict proof can be found in many text books, e.g. \cite[Theorem 8.6.5]{oksendal2013stochastic}, \cite[Theorem 27.1]{rogers2000diffusions}, \cite[Theorem 2.1]{leonard2012girsanov}.}
Since
\begin{align*}
    |\boldsymbol{b}(s, X(s)\tcb{)}|^2 &= \sum_{i=1}^N \left| \sigma^T \Lambda^{-1}\left( K {*} \bar{\rho}_s (X_i(s)) - \frac{1}{N-1} \sum_{j: j\neq i} K(X_i(s) - X_j(s)) \right) \right|^2\\
    &\leq \frac{1}{\lambda} \sum_{i=1}^N \left|  K {*} \bar{\rho}_s (X_i(s)) - \frac{1}{N-1} \sum_{j: j\neq i} K(X_i(s) - X_j(s)) \right|^2,
\end{align*}
one has by combining \eqref{eq:afterdataprocess} and \eqref{eq:aftergirsanov} that
    \begin{multline}
    D_{KL}(F^N_{[0,T]}(\cdot | z) \| \bar{F}^{N}_{[0,T]}(\cdot | z)) \leq \\
    \frac{1}{2\lambda} \sum_{i=1}^N \int_0^T \mathbb{E} \left|  K {*} \bar{\rho}_s (X_i(s)) - \frac{1}{N-1} \sum_{j: j\neq i} K(X_i(s) - X_j(s)) \right|^2 ds.
    \end{multline}
Moreover, due to the fact \eqref{eq:lawmixture} and the convexity of the KL-divergence, one has by Jensen's inequality that
\begin{equation}\label{eq:eq:afterjensen}
    D_{KL}(F^N_{[0,T]} \| \bar{F}^{\otimes N}_{[0,T]})
    \leq \frac{1}{2\lambda} \sum_{i=1}^N \int_0^T \mathbb{E} \left|  K {*} \bar{\rho}_s (X_i(s)) - \frac{1}{N-1} \sum_{j: j\neq i} K(X_i(s) - X_j(s)) \right|^2 ds,
\end{equation}
where the expectation on the right hand is now the full expectation.

Next, we estimate \eqref{eq:eq:afterjensen}. \tcb{We separately estimate this under Assumption \ref{ass:K} (bounded $K$) or Assumption \ref{ass:new} (unbounded $K$).}

\tcb{\textbf{Case 1}: Under Assumption \ref{ass:K}.}

We first split the right hand side into \eqref{eq:eq:afterjensen} \tcb{into} 
\begin{multline*}
     \sum_{i=1}^N\left|  K {*} \bar{\rho}_s (X_i(s)) - \frac{1}{N-1} \sum_{j: j\neq i} K(X_i(s) - X_j(s)) \right|^2 \\
     = \frac{1}{(N-1)^2}\sum_{i=1}^N \sum_{j: j \neq i} |A_{i,j}^\prime(s)|^2 + \frac{1}{(N-1)^2}\sum_{i=1}^N \sum_{j_1,j_2: j_1 \neq j_2, j_1 \neq i, j_2 \neq i}A_{i,j_1}^\prime(s) \cdot A_{i,j_2}^\prime(s),
\end{multline*}
where $A_{i,j}^\prime(t)$ is defined by
\begin{equation*}
    A_{i,j}^\prime(t) := K\left(X_i(t) - X_{j}(t) \right) - K {*} \bar{\rho}_t\left(X_i(t) \right). 
\end{equation*}
Since $K \in L^{\infty}$ by Assumption \ref{ass:K}, it is easy to see that for $N \geq 2$, the first term above is bounded by $8\|K\|_{\infty}^2$. For the second term, for any fixed $i$,  choosing $\rho = \rho^N_s$ (the time marginal distribution for particle position $X_s = (X_1(s) \dots X_N(s))$ at time $s$) and $\tilde{\rho} = \bar{\rho}^{\otimes N}_s$ in Lemma \ref{lmm:changemeasure} (as we shall present in Section \ref{sec:auxilliary}), for any $\eta > 0$ we have
\begin{multline*}
    \mathbb{E}\left[\frac{1}{N-1} \sum_{j_1,j_2: j_1 \neq j_2, j_1 \neq i, j_2 \neq i}A_{i,j_1}^\prime(s) \cdot A_{i,j_2}^\prime(s)\right]\\
    \leq \eta^{-1}D_{KL}\left(\rho^N_s \| \bar{\rho}^{\otimes N}_s \right) +\eta^{-1}\log  \mathbb{E}\left[\exp\left(\frac{\eta}{N-1}  \sum_{j_1,j_2: j_1 \neq j_2, j_1 \neq i, j_2 \neq i} A_{i,j_1}(s) A_{i,j_2}(s)\right)\right],
\end{multline*}
where $A_{i,j}(t)$ is defined by
\begin{equation*}
    A_{i,j}(t) := K\left(\bar{X}_i(t) - \bar{X}_{j}(t) \right) - K {*} \bar{\rho}_t\left(\bar{X}_i(t) \right). 
\end{equation*}
 Consider the map $T_s$: $Z_{[0,s]} \mapsto X_s$, by \tcb{the} data processing inequality (Lemma \ref{lmm:dp}) we know that
\begin{equation*}
    D_{KL}\left(\rho^N_s \| \bar{\rho}^{\otimes N}_s \right) \leq D_{KL}\left(F^N_{[0,s]} \| \bar{F}^{\otimes N}_{[0,s]} \right).
\end{equation*}
Also, Lemma \ref{lmm:exp_estimate} in Section \ref{sec:auxilliary} states that for $\eta \in (0,1/(4\sqrt{2}e\|K\|^2_{\infty}))$,
\begin{equation*}
    \sup_{N \geq 2, s \geq 0}\mathbb{E}\left[\exp\left(\frac{\eta}{N-1}  \sum_{j_1,j_2: j_1 \neq j_2, j_1 \neq i, j_2 \neq i} A_{i,j_1}(s) A_{i,j_2}(s)\right)\right] \leq \frac{1}{1 - 4\sqrt{2}e\|K\|^2_{\infty}\eta} < \infty.
\end{equation*}

Hence, considering the averaged summation $\frac{1}{N-1}\sum_{i=1}^N (\cdot)$ for $N \geq 2$ and combining all the above, one obtains
\begin{equation} \label{eq: kl for FN barFN}
    D_{KL}(F^N_{[0,T]} \| \bar{F}^{\otimes N}_{[0,T]}) \leq \frac{1}{2\lambda}C(\eta)T + \int_0^T \frac{1}{\lambda}\eta^{-1}D_{KL}\left(F^N_{[0,s]} \| \bar{F}^{\otimes N}_{[0,s]} \right) ds,
\end{equation}
where $C(\eta) := 8\|K\|_{\infty}^2 + \frac{2}{\eta} \log \frac{1}{1 - 4\sqrt{2}e\|K\|^2_{\infty} \eta}$. The result \eqref{eq:thm1result} is obtained after the Gr\"onwall's inequality: 
\begin{equation*}
\begin{aligned}
     D_{KL}(F^N_{[0,T]}  \| \bar{F}^{\otimes N}_{[0,T]}) &\leq \frac{C(\eta)}{2\lambda}T + \int_0^T \frac{C(\eta)}{2\lambda} \frac{1}{\lambda\eta} s e^{(\lambda\eta)^{-1}(T-s)} ds\\
     &=C(\eta)\frac{\eta}{2}\left(e^{(\lambda\eta)^{-1}T} - 1 \right) \leq Ce^{CT},
\end{aligned}
\end{equation*}
where $C$ is a positive constant independent of the particle number $N$ and the particle mass $m$. For instance, if we choose $\eta = (8\sqrt{2}e\|K\|_{\infty}^2)^{-1}$, then we can choose $C = \max(C_1, C_2)$ with $C_1 := \frac{\sqrt{2}}{4e}+\log 2$ and $C_2 := 8\sqrt{2}e\|K\|_{\infty}^2\lambda^{-1}$.

\vskip1mm
\tcb{\textbf{Case 2}: Under Assumption \ref{ass:new}.}

\tcb{Now we consider the case for the unbounded interaction kernel. First, for fixed $i$, still by Lemma \ref{lmm:changemeasure}, for any $\eta>0$, we have (recalling the notations $A_{i,j}$ and $A'_{i,j}$ above)
\begin{multline}
    \mathbb{E}\sum_{i=1}^N\left|  K {*} \bar{\rho}_s (X_i(s)) - \frac{1}{N-1} \sum_{j: j\neq i} K(X_i(s) - X_j(s)) \right|^2
    \leq \eta^{-1}D_{KL}\left(F^N_{[0,s]} \|\bar{F}^{\otimes N}_{[0,s]} \right)\\
    + \eta^{-1}\log \mathbb{E}\left[\exp\left(\eta \sum_{i=1}^N \left| K {*} \bar{\rho}_s (\bar{X}_i(s)) - \frac{1}{N-1} \sum_{j: j\neq i} K(\bar{X}_i(s) - \bar{X}_j(s))\right|^2\right)\right].
\end{multline}
Now note that 
\begin{equation}
    \mathbb{E}\left[K {*} \bar{\rho}_s (\bar{X}_i(s)) - \frac{1}{N-1} \sum_{j: j\neq i} K(\bar{X}_i(s) - \bar{X}_j(s))\right] = 0.
\end{equation}
Moreover, under Assumption \ref{ass:new}, $\bar{X}_i(s)$ is a sub-Gaussian random variable, and
\begin{equation}
    \left|K {*} \bar{\rho}_s (\bar{X}_i(s)) - \frac{1}{N-1} \sum_{j: j\neq i} K(\bar{X}_i(s) - \bar{X}_j(s)) \right| \leq C(1 + |\bar{X}_j(s)|).
\end{equation}
Therefore, the conditions required in Lemma \ref{lmm:lidu} are satisfied. Consequently, we have the similar estimate under Assumption \ref{ass:new}:
\begin{equation}
     D_{KL}(F^N_{[0,T]} \| \bar{F}^{\otimes N}_{[0,T]}) \leq \frac{CT}{2\lambda}T + \int_0^T \frac{C'}{\lambda}D_{KL}\left(F^N_{[0,s]} \| \bar{F}^{\otimes N}_{[0,s]} \right) ds,
\end{equation}
where $C$, $C'$ are positive constant independent of $N$ and $m$. Therefore the $O(1)$-upper bound for $D_{KL}(F^N_{[0,T]} \| \bar{F}^{\otimes N}_{[0,T]})$ is obtained due to Gr\"ownwall's inequality.}

Next, noting the symmetry of $F_t^N$, one has by Lemma \ref{lmm:kl} that
\begin{equation}\label{eq: kl for marginal FN barFN}
    D_{KL}\left( F^{N:k}_{[0,T]} \| \bar{F}^{\otimes k }_{[0,T]}\right) \leq \frac{k}{N}D_{KL}\left(F^N_{[0,T]} \| \bar{F}^{\otimes N}_{[0,T]} \right) \leq Ce^{CT}\frac{k}{N}.
\end{equation}
Hence, \eqref{eq:kthmarginal} holds.
\hfill
\end{proof}

\vskip1mm
The results above are all about path measures. In fact, we can extend this to the time marginal case, which is commonly studied in related literature. 

\vskip2mm
\begin{corollary}[time marginal]\label{coro:timemarginal}
     For any $t>0$, consider the distributions $F^N_t$, $\bar{F}^{\otimes N}_t$ for the second-order system defined in Section \ref{sec:intro}, with initial  $F^N_0 = \bar{F}^{\otimes N}_0.$ Then under \tcb{either Assumption \ref{ass:K} or Assumption \ref{ass:new}}, for the constant $C$ in Theorem \ref{thm:main},
    \begin{equation}
        D_{KL}(F^N_t \| \bar{F}^{\otimes N}_t) \leq Ce^{Ct},\quad \forall t > 0.
    \end{equation}
   Then for $1 \leq k \leq N$,
    \begin{equation}\label{eq:kthmarginal timemarginal}
        D_{KL}\left(F^{N:k}_t \| \bar{F}_t^{\otimes k} \right) \leq Ce^{Ct} \frac{k}{N}.
    \end{equation}
\end{corollary}
\vskip2mm

\begin{proof}
For any $t>0$, consider the path measures $F^N_{[0,t]}$, $\bar{F}^{\otimes N}_{[0,t]}$ corresponding to the time interval $[0,t]$. Then by Theorem \ref{thm:main}, 

\begin{equation*}
     D_{KL}(F^N_{[0,t]} \| \bar{F}^{\otimes N}_{[0,t]}) \leq Ce^{Ct}.
\end{equation*}
Now consider the time marginal mapping $\pi_t: C([0, t];\mathbb{R}^d)\to \mathbb{R}^d$ given by $\pi_t(Z)=Z_t$, which maps $Z$ in the path space to its time marginal $Z_t$. Then by the data processing inequality (Lemma \ref{lmm:dp}), one has
\begin{equation}
     D_{KL}(F^N_t \| \bar{F}^{\otimes N}_t) \leq D_{KL}(F^N_{[0,t]} \| \bar{F}^{\otimes N}_{[0,t]}) \leq  Ce^{Ct}.
\end{equation}
Then, \eqref{eq:kthmarginal timemarginal} is a direct result of Lemma \ref{lmm:kl}. 
\hfill
\end{proof}
\vskip2mm
\begin{remark}
The fact that the KL-divergence between path measures can control that between time marginals can actually be proved without data processing inequality, In fact, for $t>0$, the Radon-Nikodym derivative in terms of time marginal distributions has the following formula: (see, for instance, Appendix A in \cite{li2022sharp})
\begin{equation}
    \frac{d\bar{F}^{\otimes N}_t}{dF^N_t}(z) = \mathbb{E}\left[ \frac{d\bar{F}^{\otimes N}_{[0,t]}}{dF^N_{[0,t]}} \mid Z_t = z\right].
\end{equation}
Then by Jensen's inequality, we directly conclude that
\begin{equation*}
     D_{KL}(F^N_t \| \bar{F}^{\otimes N}_t) \leq D_{KL}(F^N_{[0,t]} \| \bar{F}^{\otimes N}_{[0,t]}).
\end{equation*}
In fact, these two approaches are essentially the same, since they are all due to Jensen's inequality.
\end{remark}
\vskip2mm

Based on Theorem \ref{thm:main} and Pinsker's inequality \cite{pinsker1964information}, we are able to extend the propagation of chaos to that under total variation (TV) distance defined by
\begin{equation}
     TV(\mu,\nu) := \sup_{A\in \mathcal{F}} |\mu(A) - \nu(A)|,
\end{equation}
for two probability measures $\mu$, $\nu$ defined on $(\Omega,\mathcal{F})$.
\vskip2mm
\begin{corollary}\label{coro:TV}
    Under the same settings of Theorem \ref{thm:main} and Corollary \ref{coro:timemarginal},  for $1 \leq k \leq N$ it holds that
    \begin{equation}
            TV( F^{N:k}_{[0,t]} , \Bar{F}^{\otimes k}_{[0,t]} )\le Ce^{Ct} \sqrt{\frac{k}{N}},
    \end{equation}
    for path measures and
    \begin{equation}
            TV(  F^{N:k}_t , \Bar{F}^{\otimes k}_t ) \le Ce^{Ct}\sqrt{\frac{k}{N}},
    \end{equation}
     for time marginal distributions.
\end{corollary}
\vskip2mm

\tcb{\begin{remark}
Our approach can be applied to the following first-order system without difficulty
\begin{equation}\label{eq:particle1st}
    dX_i(t) = b(X_i(t)) dt + \frac{1}{N-1} \sum_{j:j\neq i} K(X_i(t) - X_j(t)) dt + \sigma \cdot d W_i(t), \quad 1 \leq i \leq N,
\end{equation}
where $b: \mathbb{R}^d \rightarrow \mathbb{R}^d$ is the non-interaction drift and the setting of $K$, $\sigma$, $W_i$ is same as the second-order case. We skip the proof for this case.
\end{remark}}

\subsection{Some auxiliary lemmas.}\label{sec:auxilliary}

In this subsection we present some auxiliary lemmas used in our proof. The detailed proof of Lemma \ref{lmm:exp_estimate} is moved to the Appendix.

Near the end of the proof of Theorem \ref{thm:main}, in order to estimate the difference between the two drifts
\begin{equation*}
    \frac{1}{2\lambda} \sum_{i=1}^N \int_0^T \mathbb{E} \left|  K {*} \bar{\rho}_s (X_i(s)) - \frac{1}{N-1} \sum_{j: j\neq i} K(X_i(s) - X_j(s)) \right|^2 ds,
\end{equation*}
we need the following two lemmas, where a type of Fenchel-Young's inequality along with an exponential concentration estimate are needed. In fact, the Fenchel-Young type inequality (\cite[Lemma 1]{jabin2018quantitative}) states that: 
\vskip2mm
\begin{lemma}\label{lmm:changemeasure}
For any two probability measures $\rho$ and $\tilde{\rho}$ on a Polish space $E$ and some test function $F \in$ $L^1(\rho)$, one has that $\forall \eta>0$,
$$
\int_E F \rho(d x) \leq \frac{1}{\eta}\left(D_{KL}(\rho \| \tilde{\rho})+\log \int_E e^{\eta F} \tilde{\rho}(d x)\right) .
$$
\end{lemma}
\vskip2mm

We also need the following exponential concentration estimate. Similar results can be found in related literature like \cite{lim2020quantitative,jabin2018quantitative}.
For the convenience of the readers, we also attach a proof in Appendix \ref{app:expconcentration}.
\vskip2mm
\begin{lemma}\label{lmm:exp_estimate}
\tcb{Suppose Assumption \ref{ass:K} holds.} Consider solutions to the Mckean SDEs \eqref{eq:mckeansde for N} $\bar{X}_1(t)$, $\dots$, $\bar{X}_N(t)$, which are i.i.d. sampled from $\bar{F}_t$, then for fixed $\eta \in (0, 1 / \left(4\sqrt{2}e\|K\|^2_{\infty}\right))$, for any $N \geq 2$, $t \geq 0$, and $1 \leq i \leq N$ we have
\begin{equation*}
    \mathbb{E}\left[\exp\left(\frac{\eta}{N-1} \sum_{j_1,j_2: j_1 \neq j_2, j_1 \neq i, j_2 \neq i} A_{i,j_1}(t) \cdot A_{i,j_2}(t)\right) \mid \bar{X}_i(t)\right]
    \leq \frac{1}{1 - 4\sqrt{2}e\|K\|_{\infty}^2\eta } < + \infty,
\end{equation*}
where $A_{i,j}(t)$ is defined by
\begin{equation*}
    A_{i,j}(t) := K\left(\bar{X}_i(t) - \bar{X}_{j}(t) \right) - K {*} \bar{\rho}_t\left(\bar{X}_i(t) \right). 
\end{equation*}
\end{lemma}
\vskip2mm

\tcb{When the interaction kernel $K$ is bounded}, Lemma \ref{lmm:changemeasure}, Lemma \ref{lmm:exp_estimate} along with other previous analysis enable one to obtain an $\mathcal{O}(1)$-upper bound for  $D_{KL}(F^N_{[0,T]} \| \bar{F}^{\otimes N}_{[0,T]})$, and it is easy to see that the bound is independent of the particle mass $m$. \tcb{When $K$ is not bounded, we make use of Lemma \ref{lmm:changemeasure} and Lemma \ref{lmm:lidu} below instead:}
\vskip2mm

\tcb{\begin{lemma}\cite[Lemma 3.3]{du2024collision},\label{lmm:lidu}
Consider $\rho \in \mathcal{P}(E)$ and $\psi(x)$ satisfying $\int_E \psi(x) \rho(d x)=0$ and for the universal constant $c_*>0$ in the Hoeffding's inequality, the following holds
\begin{equation}
\|\psi(x)\|_\rho:=\inf \left\{c>0: \int_E \exp \left(|\psi(x)|^2 / c^2\right) \rho(d x) \mid \leq 2\right\}<c_*.
\end{equation}
Then,
\begin{equation}
\sup _{N \geq 1} \int_{E^N} \exp \left(\frac{1}{N}\left|\sum_{i=1}^N \psi\left(x_i\right)\right|^2\right) \rho^{\otimes N} \mathrm{dx}<\infty .
\end{equation}
\end{lemma}}

\tcb{For readers' convenience, here we briefly introduce the Hoeffding bound used in the statement (as well as its proof) of Lemma \ref{lmm:lidu} above. The Hoeffding inequality \cite{vershynin2018high} claims that for $n$ independent centered real random variables $Y_1, \dots, Y_n$, there exists a universal constant $C_* > 0$ such that}
\begin{equation}
    \tcb{P\left(\left|\sum_{j=1}^n Y_j \right| \geq y \right) \leq 2\exp\left(-\frac{c_* y^2}{\sum_{j=1}^n \| Y_j \|_{\psi_2}^2} \right),\quad \forall\, y \geq 0,}
\end{equation}
\tcb{where the $\psi_2$ norm (or the Orlicz norm with $\psi_2(x) = \exp(x^2)-1$) for some sub-Gaussian random variable $X$ is given by}
\begin{equation}
    \tcb{\|X\|_{\psi_2} := \inf \left\{c > 0:\mathbb{E}\left[\exp(|x|^2/c^2)\right] \leq 2 \right\}.}
\end{equation}

The following well-known linear scaling property of the relative entropy is useful for controlling the marginal distribution. \tcb{(See e.g. \cite[Lemma 3.9]{miclo2001genealogies}, \cite[Equation (2.10), page 772]{csiszar1984sanov}.)}
\vskip2mm
\begin{lemma}[linear scaling for KL-divergence]\label{lmm:kl}
Let $\mu^n \in \mathcal{P}_s(E^{n})$ be a symmetric distribution over some space tensorized space $E^n$ and $\bar{\mu} \in \mathcal{P}(E)$.
For $1 \leq k \leq n$, define its $k$-th marginal $\mu^{n:k}$ by
\begin{equation}
    \mu^{n:k}(z_1,\dots,z_k) := \int_{E^{n-k}} \mu^N(z_1,\dots,z_n)dz_{k+1}\dots dz_n.
\end{equation}
Assume that $\mu^{n:k}\ll \bar{\mu}^{\otimes k}$ for any $1\le k \le N.$
Then it holds that
\begin{equation}\label{eq:lmm2}
    D_{KL}\left(\mu^{n:k} \| \bar{\mu}^{\otimes k}\right) \leq 2\frac{k}{n}D_{KL}\left(\mu^{n} \| \bar{\mu}^{\otimes n}\right).
\end{equation}
\end{lemma}

\section{Other applications}\label{sec:discussion}
\tcb{In this section, we show two application of our approach in neural networks and numerical analysis respectively.}

\subsection{\tcb{Application in neural networks.}} \tcb{An interesting application is on neural networks. To show the characteristics of our approach, we use an artificial single-layer neural network as an example: 
\begin{equation}\label{eq:NN1}
    X_i(T) = \mathfrak{S} \Bigg(\int_0^T b(X_i(t)) dt + \frac{1}{N-1} \sum_{j:j\neq i} K(X_i(t) - X_j(t)) dt + \sigma \cdot d W_i(t)\Bigg), \quad 1 \leq i \leq N,
\end{equation}
where $\{X_i(0)\},$ $i=1,\cdots, N$ denotes $N$ input features and $\mathfrak{S}$ denotes certain activate function. The $\{X_i(T)\},$ $i=1,\cdots, N$ means the output. This model can be viewed as a single-layer variant with noise from the model mentioned in \cite{wang2023acmp}. Our approach can be directly applied into \eqref{eq:NN1} and transform the original problem in the space of $X$ into the space of the driving process
$$
\theta_i(t) =  \int_0^t \left(  \frac{1}{N-1} \sum_{j: j\neq i} K(X_i(s) - X_j(s)) - K {*} \bar{\rho}_s (X_i(s))\right) ds +\sigma\cdot W_i(t),
$$
similarly to the discussion in Section \ref{sec:sec2}. The existence of the activate function $\mathfrak{S}$ make it impossible to use Girsanov's theorem directly, while our approach works in this case as well. Also, if one uses the second-order dynamics to update the features, that is,
\begin{equation*}
    X_i(T) = \mathfrak{S} \Bigg(\int_0^T V_i(t))\Bigg), \quad V_i(t) \text{is obtained by \eqref{eq:particle}},
\end{equation*}
the uniformity in mass is not a direct byproduct of Girsanov's theorem.}

\subsection{\tcb{Application in numerical analysis.}}
\tcb{Our approach can be applied in numerical analysis directly. For example, take the following scheme of SDE \eqref{eq:particle} with time step $h$. Without loss of generality, we set $m=1$ and $\sigma=1$. Assume that $K$ is globally Lipschitz continuous with a constant $C_K$, and the second moment of the initial data is finite:
\begin{equation}\label{eq:numini}
    \mathbb{E}|Z(0)|^2<\infty.
\end{equation}
Define 
\begin{equation*}
    Z:= \begin{pmatrix}
        X\\V
    \end{pmatrix},\quad
    A:= \begin{pmatrix}
        0&1\\
        0&-\gamma
    \end{pmatrix},\quad
    B(X(t)):=\begin{pmatrix}
        0\\
        \frac{1}{N-1}\sum\limits_{j:j\ne i}K(X_i(t)-X_j(t))
    \end{pmatrix},\quad
     C:=\begin{pmatrix}
       0\\1
    \end{pmatrix}.
\end{equation*}
We use $\tilde{Z},\tilde{X},\tilde{V}$ to denote the numerical solution. For $t\in[t_k,t_{k+1)},$ ($t_k =kh$), $\tilde{Z}$ is defined by 
\begin{equation*}
    \tilde{Z}_t = e^{A(t-t_k)}\tilde{Z}(t_k) + \int_{t_k}^t e^{A(t-s)}B(\tilde{X}(t_k))ds + \int_{t_k}^t e^{A(t-s)}C dW_s.
\end{equation*}
For $T:=nh$ and $\tilde{F}^N_{[0,T]}:=\text{Law}(\tilde{Z}),$ similar to the proof of Theorem \ref{thm:main}, one has
\begin{equation}\label{eq:num1}
\begin{aligned}
      D_{KL}(\tilde{F}^N_{[0,T]}\|F^N_{[0,T]})&\le \mathbb{E}\sum\limits_{k=0}^{n-1}\int_{t_k}^{t_{k+1}} \sum_{\substack{i,j=1,\\ j\ne i}}^N \frac{1}{N-1}|K(\tilde{X}_i(t)-\tilde{X}_j(t))-K(\tilde{X}_i(t_k)-\tilde{X}_j(t_k))|^2dt\\&
    \le   C\mathbb{E} N  \sum\limits_{k=0}^{n-1}\int_{t_k}^{t_{k+1}}|K(\tilde{X}_1(t)-K(\tilde{X}_1(t_k)|^2dt.
\end{aligned}
\end{equation}
Consider equation \eqref{eq:num1}, by It\^o's calculus and the assumption on $K$, one has 
\begin{equation*}
    \begin{aligned}
        d\mathbb{E} |\tilde{V}_i|^2=& 2\mathbb{E}\tilde{V}_i \cdot \left( \frac{1}{N-1}\sum\limits_{j: j\ne i} K(\tilde{X}_i(t_k) - \tilde{X}_j(t_k))dt - \gamma \tilde{V}_i(t)dt\right) +d\,dt\\
        \le& C \bigg(|K(0)|\mathbb{E}|\tilde{V}_i| + \mathbb{E}\,|\tilde{V}_i| \cdot \frac{1}{N-1}\sum\limits_{j: j\ne i}  |\tilde{X}_j(t_k) - \tilde{X}_i(t_k)| + \mathbb{E}|\tilde{V}_i|^2 +d \bigg)dt\\
        \le& C(\mathbb{E}|\tilde{X}_i|^2 + \mathbb{E}|\tilde{X}_j|^2 + \mathbb{E}|\tilde{V}_i|^2 +1).
    \end{aligned}
\end{equation*}
By the exchangeability, $\mathbb{E}|\tilde{X}_i|^2 = \mathbb{E}|\tilde{X}_j|^2.$ One has
\begin{equation*}
    d\mathbb{E} |\tilde{V}_i|^2\le C(\mathbb{E}|\tilde{X}_i|^2 + |\tilde{V}_i|^2 +1).
\end{equation*}
By the Gr\"onwall inequality and the assumption \eqref{eq:numini}, it holds that
\begin{equation}
    \mathbb{E}|\tilde{V}_i(t)|^2 < \infty, \quad \forall t\in [0,T].
\end{equation}
Hence,
\begin{equation}\label{eq:num2}
    \mathbb{E}\,|\tilde{X}_1(t)-\tilde{X}_1(t_k)|^2 \le C \mathbb{E}\,\big|\int_{t_k}^t \tilde{V}_1(s)ds\big|^2 \le \mathbb{E}\,\sup\limits_{s\le t}|\tilde{V}_1(s)|^2 h^2 \le Ch^2.
\end{equation}
Then, combining \eqref{eq:num1} and \eqref{eq:num2}, one obtains
\begin{equation}
\begin{aligned}
      D_{KL}(\tilde{F}^N_{[0,T]}\|F^N_{[0,T]})&
    \le   CC_KN\mathbb{E}   \sum\limits_{k=0}^{n-1}\int_{t_k}^{t_{k+1}}|\tilde{X}_1(t)-\tilde{X}_1(t_k)|^2dt\\
    &\le CNh^2.
\end{aligned}
\end{equation}}

\section{More discussions}\label{sec:dis}
Here we present brief discussions on the reversed relative entropy and the mass independence phenomenon.
\subsection{Discussion on the reversed relative entropy.}
In section \ref{sec:mainresult}, we estimated the relative entropy $D_{KL}(F^N_{[0,T]} \| \bar{F}^{\otimes N}_{[0,T]})$.
If we consider the reversed relative entropy, by the data processing inequality, one would obtain that
\begin{equation}
D_{KL}( \bar{F}^{\otimes N}_{[0,T]} \| F^N_{[0,T]})
\le D_{KL}(Q^2\| Q^1)=-\mathbb{E} \log \frac{d Q^1}{dQ^2}(\theta^{(2)}).
\end{equation}
Since
\[
\pi_s\circ \Phi_s(\theta^{(2)})=\bar{X}(s),
\]
one thus finds that
\[
D_{KL}(Q^2\| Q^1)=\mathbb{E} \sum_i \int_0^t |\boldsymbol{b}_i(s, \bar{X}(s))|^2\,ds.
\]
Here, $\bar{X}=(\bar{X}_1,\cdots, \bar{X}_N)$ is the position process for the mean-field McKean SDE, whose components are i.i.d..
Hence, the right hand side can be estimated by
\begin{equation}\label{eq:reversed}
    D_{KL}(Q^2\| Q^1) \le C\frac{T}{\lambda},
\end{equation}
where $C$ is independent of $T$ and $N$. \tcb{The result linearly depending on $T$ is similar with \cite[Lemma 4.11]{lacker2023hierarchies}.} This is an interesting observation, though the consequence of such a relative entropy estimate is unclear.

\subsection{Discussion on the mass-independence.}

Denote the marginal distributions in the $v$-direction:
    \begin{equation}
        \mu_v^N(v):= \int_{\mathbb{R}^{Nd}} F^N dx ,\quad \bar{\mu}_v(v):= \int_{\mathbb{R}^d} \bar{F} dx.
    \end{equation}
It is not difficult to see from the proof of Theorem \ref{thm:main} that the KL-divergence $D_{KL}\left(\mu_v^N \| \bar{\mu}_v^{\otimes N} \right)$ in the $v$-direction has an $\mathcal{O}(1)$ upper-bound, and the bound is \textbf{independent of the particle mass $m$}. \tcb{The mass-independence result is particularly interesting from a physical perspective. Additionally, when conducting numerical simulations in the regime of large friction, such as in viscous fluids, this phenomenon must be taken into account. Some researchers \cite{wang2024small,carrillo2021mean,wang2022small} focus on the zero mass limit under various conditions. If the propagation of chaos can be shown to be uniform in mass, then the result is asymptotically preserving in the overdamped limit.}

\tcb{However, the mass independence result is not very natural from a physical perspective.  For fixed mass $m$ and fixed initial data, considering the mapping $\varphi^m_T: \theta \rightarrow V,$
the limiting behavior as  $m \rightarrow 0$ is poor and the $L^2$ norm of $V^N$ (or $\bar{V}^{\otimes N}$) usually diverges. On the other hand, under our framework, the dependence of $m$ in the mapping $\Phi$ is not important when applying the data processing inequality.  This may indicates the KL divergence is a suitable tool to obtain a rate independent of the mass. To illustrate this, we provide a simple example.} Consider the channel $\Psi^m(X) := X + Z_m$, where $Z_m \sim \mathcal{N}(0,m^{-2})$. 
Then, if we simply consider the Gaussian data $X \sim \mathcal{N}(0,1)$, $Y \sim \mathcal{N}(1,1)$, the inequality for the KL-divergence between their distributions $\mu_X$, $\mu_Y$ still holds for any $m$: $D_{KL}(\text{Law}(\Psi^m(X)) \| \text{Law}(\Psi^m(Y))\leq D_{KL}(\mu_X \| \mu_Y)$. In fact, direct calculation gives $D_{KL}(\mu_X \| \mu_Y) = \frac{1}{2}$, and $D_{KL}(\text{Law}(\Psi^m(X)) \| \text{Law}(\Psi^m(Y)) = \frac{1}{2(1 + m^{-2})}$, since $\Psi^m(X) \sim \mathcal{N}(0,1 + m^{-2})$,  $\Psi^m(Y) \sim \mathcal{N}(1,1 + m^{-2})$. However, it is easy to check that the $L^2$ norm of single data may blow up as $m$ tends to zero, since the variance of $\Psi^m(X)$ is just $1 + m^{-2}$.

\section*{Acknowledgement}
This work is financially supported by the National Key R\&D Program of China, Project Number 2021YFA1002800 and Project Number 2020YFA0712000. The work of L. Li was partially supported by NSFC 12371400 and 12031013,  Shanghai Science and Technology Commission (Grant No. 21JC1403700, 20JC144100, 21JC1402900), the Strategic Priority Research Program of Chinese Academy of Sciences, Grant No. XDA25010403, Shanghai Municipal Science and Technology Major Project 2021SHZDZX0102. We thank Zhenfu Wang \tcb{and the anonymous referees} for some helpful comments.

\appendix


\section{Basics on path measure and Girsanov's transform}\label{app:girsanov}


\tcb{Here we present a formal derivation of Girsanov's transform. Note that the derivation here is never meant to be a proof. We present it here for the convenience of readers for intuitive understanding.} Consider the following two SDEs in $\mathbb{R}^d$ with different predictable drifts but the same diffusion $\sigma$, \tcb{which we assume are weakly well-posed}.
\begin{equation}\label{eq:A1}
\left\{
\begin{aligned}
     X^{(1)}_t & = x_0 +  \int_0^t b^{(1)}\left(s, [X^{(1)}_{[0,s]}]\right) ds+ \int_0^t \sigma \cdot d W_s,\,t\le T,\\
     X^{(2)}_t & = x_0 +  \int_0^t b^{(2)}\left(s, [X^{(2)}_{[0,s]}]\right) ds+ \int_0^t \sigma \cdot d W_s,\,t\le T.
\end{aligned}
\right.
\end{equation}
Here $W$ is a standard Brownian motion under the probability measure $\mathbb{P}$ (the same for the two systems), and $x_0\sim \mu_0$ is a common, but random, initial position. Here, the drift $b^{(i)}(s, [\gamma_{[0, s]}])$ depends on the path $\gamma_{\tau}$ for $0\le \tau\le s$. 

For a fixed time interval $[0,T],$ the two processes $X^{(1)}$ and $X^{(2)}$ naturally induce two probability measures in the path space $\mathcal{X}^\prime :=C([0,T],\mathbb{R}^d)$,
denoted by $P^{(1)}$ and $P^{(2)},$ respectively. 

Define the process
\begin{equation}
 u\left(X^{(2)}_{[0,t]}\right) = \sigma^T \Lambda^{-1} \left(b^{(2)}-b^{(1)}\right)\left(X^{(2)}_{[0,t]}\right),
\end{equation} 
where $\Lambda = \sigma \sigma^T$. By Girsanov theorem, under the probability measure $\mathbb{Q}$ satisfying
\begin{equation}\label{eq:appGirs1}
      \frac{d\mathbb{Q}}{d\mathbb{P}}(\omega) = \exp\Big(\int_0^T   -u\left(X^{(2)}_{[0,s]}\right) \cdot dW_s   - \frac{1}{2} \int_0^T \left|u\left(X^{(2)}_{[0,s]}\right)\right|^2  ds\Big),
\end{equation}
the law of $X^{(2)}$ is the same as the law of $X^{(1)}$ under $\mathbb{P}$. In other words, for any Borel measurable set $B\subset \mathcal{X}^\prime,$
\[
\mathbb{E}_{\mathbb{P}}[\textbf{1}_B(X^{(1)}(\omega))]
=\mathbb{E}_{\mathbb{Q}}[\textbf{1}_B(X^{(2)}(\omega))]
=\mathbb{E}_{\mathbb{P}}\left[\textbf{1}_B(X^{(2)})\frac{d\mathbb{Q}}{d\mathbb{P}}(\omega)\right].
\]
Since $P^{(1)} =(X^{(1)})_{\#}\mathbb{P}$ and $P^{(2)}
=(X^{(2)})_{\#}\mathbb{P}$ are the laws of $X^{(1)}$ and $X^{(2)}$ respectively, then one has
\[
P^{(1)}(B)=\mathbb{E}_{X\sim P^{(2)}}\left[\textbf{1}_B(X) \frac{dP^{(1)}}{dP^{(2)}}(X) \right]=  \mathbb{E}_{\mathbb{P}}\left[\textbf{1}_B(X^{(2)}(\omega)) \frac{dP^{(1)}}{dP^{(2)}}(X^{(2)}(\omega))\right].
\]
It follows that the Radon-Nikodym derivative satisfies 
\begin{equation}\label{eq:rnderavative}
    \frac{dP^{(1)}}{dP^{(2)}}(X^{(2)}(\omega)) = \frac{d\mathbb{Q}}{d\mathbb{P}}(\omega)=\exp\Big(\int_0^T   -u\left(X^{(2)}_{[0,s]}\right) \cdot dW_s   - \frac{1}{2} \int_0^T \left|u\left(X^{(2)}_{[0,s]}\right)\right|^2  ds\Big),\,a.s.,
\end{equation}
which is a martingale under $\mathbb{P}$ and its natural filtration $\mathcal{F}_{t}^{(2)}:=\sigma (X_s^{(2)},s\le t),$ $t\in[0,T].$

Below, for the reader's convenience, we give a simple derivation for the formulas \eqref{eq:appGirs1} (or \eqref{eq:rnderavative}) from a discrete perspective. This is not a rigorous proof but it is illustrating for the Girsanov's transform. For simplicity, let $d = d'$ and $\sigma \in \mathbb{R}_{+}$ be a scalar. The general derivation can be performed similarly. 

Consider 
\begin{equation*}
    X_{n+1}^{(1)} = X_n^{(1)} + b^{(1)}_{n} \tau + \sqrt{\tau} \sigma Z_n,\quad X_0^{(1)} = x_0 \sim f_0,
\end{equation*}
where  $b^{(1)}_{n} := b^{(1)}(s, [\tilde{\gamma}]_{[0, s]})$, where $\tilde{\gamma}_s$ is some interpolation
using the data $X_0^{(1)}, \cdots, X_n^{(1)}$,  and $Z_n \sim N(0,I_d)$ under probability measure $\mathbb{P}$. 

Clearly the posterior distribution $f(X_i^{(1)} \mid X_0^{(1)},\dots X_{i-1}^{(1)})$ is Gaussian, so one can calculate the joint distribution $f(x_0^{(1)}, \dots, x_N^{(1)})$ of $(X_0^{(1)}, \dots X_N^{(1)})$:
\begin{equation*}
    f(x_0^{(1)}, \dots, x_N^{(1)}) = \left(2\pi \tau \sigma^2 \right)^{-\frac{N}{2}} \exp \left(-\frac{1}{2\tau\sigma^2} \sum_{i=1}^N \left| x_i^{(1)} - x_{i-1}^{(1)} - b^{(1)}_{{i-1}} \tau \right|^2 \right) f_0.
\end{equation*}
Suppose there is another probability measure $\mathbb{Q}$ such that the law of $X^{(1)}$ is the same as the law of $X^{(2)}$ under $\mathbb{Q}$, where one can similarly introduce the discrete version
\begin{equation*}
    X_{n+1}^{(2)} = X_n^{(2)} + b^{(2)}_{n} \tau + \sqrt{\tau} \sigma Z_n,\quad X_0^{(2)} = x_0 \sim f_0,
\end{equation*}
and the joint distribution
\begin{equation*}
    \tilde{f}(x_0^{(2)}, \dots, x_N^{(2)}) = \left(2\pi \tau \sigma^2 \right)^{-\frac{N}{2}} \exp \left(-\frac{1}{2\tau\sigma^2} \sum_{i=1}^N \left| x_i^{(2)} - x_{i-1}^{(2)} - b^{(2)}_{{i-1}} \tau \right|^2 \right) f_0.
\end{equation*}
Then by change of measure, for any measurable $F$, it holds
\begin{equation*}
    \int F(X)\frac{d\mathbb{Q}}{d\mathbb{P}} d\mathbb{P} = \int F(X) d\mathbb{Q},
\end{equation*}
namely,
\begin{multline*}
    \int F(x_0, \dots, x_N) f(x_0 ,\dots, x_N) \frac{d\mathbb{Q}}{d\mathbb{P}} \circ X^{-1} (x_0, \dots, x_N) dx_0 \dots dx_N\\
    = \int F(x_0, \dots, x_N) \tilde{f}(x_0, \dots, x_N)dx_0 \dots dx_N.
\end{multline*}
So clearly $\frac{d\mathbb{Q}}{d\mathbb{P}} = \lim\limits_{\tau \rightarrow 0} L^{-1}(\tau)$, where
\begin{equation*}
\begin{aligned}
    L(\tau) &= \frac{f}{\tilde{f}} = \exp \left(-\frac{1}{2\tau\sigma^2}\sum_{i=1}^N \left(\left(x_i - x_{i-1} - b^{(1)}_{{i-1}} \tau\right)^2 - \left(x_i - x_{i-1} - b^{(2)}_{{i-1}}\tau \right)^2 \right) \right)\\
    & = \exp \left(-\frac{1}{2\tau \sigma^2} \sum_{i=1}^N \left(2 \tau (x_i - x_{i-1}) \cdot (b^{(2)}_{{i-1}} - b^{(1)}_{{i-1}}) + \tau^2 \left( |b^{(1)}_{{i-1}}|^2 - |b^{(2)}_{{i-1}}|^2\right)\right) \right).
\end{aligned}
\end{equation*}
Letting $\tau \rightarrow 0$, we are expected to have
\begin{multline*}
    \lim_{\tau \rightarrow 0} L^{-1}(\tau) = \exp\left( \frac{1}{\sigma^2} \left( \int_0^t (b^{(2)} - b^{(1)})(s, [X_{[0,s]}]) \cdot dX_s\right. \right.\\
    +\left. \left.\frac{1}{2} \int_0^t\left(|b^{(1)}|^2(X_{[0,s]}) - |b^{(2)}_s|^2(X_{[0,s]})\right) ds \right)\right).
\end{multline*}
Taking into account $X \sim P^{(1)}$ (recall $P^{(i)} = X^{(i)}_\# \mathbb{P}$, $i=1,2$), we derive that
\begin{multline*}
    \frac{dP^{(2)}}{dP^{(1)}} (X^{(1)}) = \exp \left(\frac{1}{\sigma} \int_0^t (b^{(2)} - b^{(1)})\left(s, [X^{(1)}]_{[0,s]}\right) \cdot dW_s\right.\\
    \left. - \frac{1}{2\sigma^2} \int_0^t |b^{(2)} - b^{(1)}|^2 \left(s, [X^{(1)}]_{[0,s]}\right) ds \right).
\end{multline*}
Also, since the two measures $P^{(1)}$, $P^{(2)}$ are equivalent, $\frac{dP^{(1)}}{dP^{(2)}}$ is well defined and can be derived in the exactly same way. Here we directly present its expression
\begin{multline*}
    \frac{dP^{(1)}}{dP^{(2)}} (X^{(2)}) = \exp \left(\frac{1}{\sigma} \int_0^t (b^{(1)} - b^{(2)})\left(s, [X^{(2)}]_{[0,s]}\right) \cdot dW_s \right.\\
    \left. - \frac{1}{2\sigma^2} \int_0^t |b^{(2)} - b^{(1)}|^2 \left(s, [X^{(2)}]_{[0,s]}\right) ds \right).
\end{multline*}

\section{Proof of Lemma \ref{lmm:exp_estimate}}
\label{app:expconcentration}
Here we prove Lemma \ref{lmm:exp_estimate} in Section \ref{sec:auxilliary}. The critical point of the proof is the usage of he Marcinkiewicz-Zygmund type inequality (see for instance, Theorem 2.1 in \cite{rio2009moment}, Lemma 5.2 in \cite{lim2020quantitative}, or Lemma 3.3 in \cite{li2023solving}).
\vskip2mm
\begin{proof}({\bf Proof of Lemma \ref{lmm:exp_estimate}.})
Fix $i$ and fix $t>0$. For $1 \leq k \leq N$ define
\begin{equation*}
    D_k := \sum_{j: j<k,j \neq i}A_{i,k}(t)\cdot A_{i,j}(t).
\end{equation*}
Then 
\begin{equation*}
    \sum_{j_1,j_2: j_1 \neq j_2, j_1 \neq i, j_2 \neq i} A_{i,j_1}(t) \cdot A_{i,j_2}(t) = 2\sum_{k: k\neq i}D_k.
\end{equation*}
Clearly, since $\mathbb{E}\left[A_{i,j_1}(t) \cdot A_{i,j_2}(t) \mid \bar{X}_i(t)\right] = \mathbb{E}\left[A_{i,j_1}(t)  \mid \bar{X}_i(t)\right] \cdot \mathbb{E}\left[ A_{i,j_2}(t) \mid \bar{X}_i(t)\right] = 0 $ ($j_1 \neq j_2$, $j_1 \neq i$, $j_2 \neq i$) by independency, and since $|A_{i,j}(t)|$ is uniformly upper-bounded by $2\|K\|_{\infty}$  by Assumption \ref{ass:K}, we know that $(D_k)_k$ is a sequence of $L^p$-martingale differences ($p \geq 2$) with respect to the filtration $\mathcal{F}_k := \sigma \left(\bar{X}_1(t), \dots \bar{X}_k(t);\bar{X}_{i}(t) \right).$ That is, for each $k\ge 1,$ $D_k$ is $\mathcal{F}_k$-measurable, $D_k \in L^p$ and $\mathbb{E}\left[D_k \mid \mathcal{F}_{k-1}\right] = 0.$ This then enables one to apply the Marcinkiewicz-Zygmund type inequality, and to obtain
\begin{equation*}
    \| \sum_{k:k\neq i} D_k \|_{L^p}^2 \leq (p-1) \sum_{k:k\neq i} \| D_k \|_{L^p}^2,\quad \forall p \geq 2.
\end{equation*}
Moreover, for each $k \neq i$, define the sequence
\begin{equation*}
    B_j^k = A_{i,k}(t) \cdot A_{i,j}(t), \quad j < k, j \neq i.
\end{equation*}
Clearly, $D_k = \sum_{j: j<k,j\neq i}B_j^k$, $(B_j^k)_j$ is a sequence of $L^p$-martingale differences ($p\geq 2$) with respect to the filtration $\hat{\mathcal{F}}_j := \sigma\left(\bar{X}_1(t), \dots \bar{X}_j(t); \bar{X}_k(t),\bar{X}_i(t) \right)$, and $\mathbb{E}\left[B_j^k \mid \hat{\mathcal{F}}_{j-1}\right] = 0$. Using the Marcinkiewicz-Zygmund type inequality again, one obtains
\begin{equation*}
    \|D_k\|_{L^p}^2 \leq (p-1) \sum_{j:j<k,j\neq i}\| B^k_j\|_{L^p}^2.
\end{equation*}
Now  Taylor's expansion gives
\begin{equation*}
\begin{aligned}
    \mathbb{E}\Biggl[\exp\biggl(\frac{2\eta}{N-1}\sum_{k:k\neq i} D_k \biggr)\mid \Biggr.&\Biggl. \bar{X}_i(t)\Biggr] = 1 + \sum_{p=2}^{\infty}\frac{(2\eta^p)}{p!(N-1)^{p}} \| \sum_{k:k\neq i} D_k \|_{L^p}^p\\
    &\leq 1 + \sum_{p=2}^{\infty}\frac{(2\eta)^p (p-1)^{\frac{p}{2}}}{p!(N-1)^{p}}\left(\sum_{k:k\neq i} \| D_k\|_{L^p}^2 \right)^{\frac{p}{2}}\\
    &\leq 1 + \sum_{p=2}^{\infty}\frac{(2\eta)^p (p-1)^{\frac{p}{2}}}{p!(N-1)^{p}}\left(\sum_{k:k\neq i} (p-1) \sum_{j:j<k,j\neq i}\| B^k_j\|_{L^p}^2 \right)^{\frac{p}{2}}\\
    & \le 1 + \sum_{p=2}^{\infty} \left(4\sqrt{2} \|K \|_{\infty}^2 \eta \right)^p \frac{(p-1)^p}{p!} \left( \frac{N-2}{N-1}\right)^{\frac{p}{2}}.
\end{aligned}
\end{equation*}
Note that all $L^p$ norm above is associated with the conditional expectation 
$\mathbb{E}\left[\cdot \mid \bar{X}_i(t)\right]$. For $N \geq 2$, $ \frac{N-2}{N-1} < 1$. Moreover, by Stirling's formula, there exists $\theta_p \in (0,1)$ such that
\begin{equation*}
     \frac{(p-1)^{p}}{p!} = \frac{(p-1)^{p} e^p e^{-\frac{\theta_p}{12p}}}{p^p \sqrt{2\pi p}} \leq e^p,\quad \forall p \geq 2.
\end{equation*}
Hence, if we choose $\eta \in (0,1/(4\sqrt{2}e\|K\|_{\infty}^2))$,
\begin{equation*}
      \mathbb{E}\left[\exp\left(\frac{2\eta}{N-1}\sum_{k:k\neq i} D_k \right)\mid \bar{X}_i(t)\right] \leq 1 + \sum_{p=2}^{\infty} \left(4\sqrt{2}e\|K\|_{\infty}^2\eta \right)^p \leq \frac{1}{1 - 4\sqrt{2}e\|K\|_{\infty}^2\eta} < + \infty.
\end{equation*}
\hfill
\end{proof}

\bibliographystyle{plain}
\bibliography{main}

\end{document}